\documentclass[12pt,twoside,final,psamsfonts]{amsart}
\usepackage[psamsfonts]{amssymb}
\usepackage[dvips,final]{graphicx,psfrag}
\usepackage{times,a4wide}

\theoremstyle{plain}
\newtheorem*{theorem*}{Theorem}

\newtheorem{theorem}{Theorem}[section]
\newtheorem{proposition}[theorem]{Proposition}
\newtheorem*{proposition*}{Proposition}
\newtheorem*{mainT}{Main Theorem}
\newtheorem{corollary}[theorem]{Corollary}
\newtheorem*{corollary*}{Corollary}
\newtheorem{lemma}[theorem]{Lemma}
\newtheorem*{lemma*}{Lemma}

\theoremstyle{remark}
\theoremstyle{definition}
\newtheorem{remark}[theorem]{Remark}
\newtheorem*{remark*}{Remark}

\theoremstyle{definition}
\newtheorem{definition}[theorem]{Definition}
\newtheorem*{definition*}{Definition}

\begin{document}

\newcommand{\lil}{\lambda\in \Lambda}
\newcommand{\D}{\mathbb{D}}
\newcommand{\C}{\mathbb{C}}
\newcommand{\N}{\mathcal{N}}
\newcommand{\R}{\mathbb{R}}
\newcommand{\dist}{\operatorname{dist}}
\newcommand{\Int}{\operatorname{Int}}
\newcommand{\Hol}{\operatorname{Hol}}
\newcommand{\Har}{\operatorname{Har}}
\renewcommand{\Re}{\mbox{Re}}
\renewcommand{\Im}{\mbox{Im}}

\renewcommand{\qedsymbol}{$\blacksquare$}

\title{Traces of the Nevanlinna class on discrete sequences}

\author{A. Hartmann, X. Massaneda, A. Nicolau }

\address{A.Hartmann:   Universit\'e de Bordeaux\\
 IMB\\ 351 cours de la Lib\'era\-tion\\ 33405 Talence\\ France}
 \email{Andreas.Hartmann@math.u-bordeaux1.fr}  

\address{X.Massaneda: Universitat  de Barcelona\\
Departament de Matem\`a\-tiques i Inform\`atica and BGSMath \\
Gran Via 585, 08007-Bar\-ce\-lo\-na\\ Spain}
\email{xavier.massaneda@ub.edu}

\address{A.Nicolau: Universitat Aut\`onoma de Barcelona\\
Departament de Matem\`a\-tiques\\
Edifici C, 08193-Bellaterra\\ Spain}
\email{artur@mat.uab.cat}


\thanks{Second and third authors supported by the Generalitat de Catalunya (grants 2014 SGR 289 and 2014SGR 75) and the Spanish Ministerio de Ciencia e Innovaci\'on (projects MTM2014-51834-P and  MTM2014-51824-P)}

\date{\today}


\begin{abstract} 
We show that a discrete sequence $\Lambda$ of the unit disk is the union of
$n$ interpolating sequences for the Nevanlinna class $\N$ if and only if
the trace of  $\N$ on $\Lambda$ coincides with the space of functions on
$\Lambda$ for which the  divided differences of order $n-1$ are uniformly
controlled by a positive harmonic function. 
\end{abstract}

\maketitle

\section{Definitions and statement}
This note deals with some properties of the classical
\emph{Nevanlinna class} consisting of the holomorphic functions in the unit disk $\D$ for which $\log_+|f|$ has a positive harmonic majorant. We denote by $\Har_+(\D)$ the set of non-negative harmonic functions in $\D$. Equivalently, 
\[
\N=\bigl\{f\in \Hol(\D):\lim_{r \to 1}\frac{1}{2\pi}\int_{0}^{2\pi}
       \log^+|f(re^{i\theta})|\;d\theta<\infty\bigr\}.
\]

\begin{definition*}
A discrete sequence of points $\Lambda$ in $\D$ is called \emph{interpolating for} $N$ (denoted $\Lambda\in\Int \N$) if the trace space $N|\Lambda$ is ideal, or equivalently, if for every $v\in\ell^\infty$ there exists $f\in \N$ such that 
\[
 f(\lambda_n)=v_n,\quad n\in\mathbb N.
\]
\end{definition*}

Interpolating sequences for the Nevanlinna class were first investigated by Naftalevitch \cite{Naft}. A rather complete study was carried out much later in \cite{HMNT}. Let $B$ denote the Blaschke product associated to a Blaschke sequence $\Lambda$. Let 
\[
 b_\lambda(z)=\frac{z-\lambda}{1-\bar\lambda z}\quad\textrm{and}\quad B_\lambda(z)=\frac{B(z)}{b_\lambda(z)}\ .
\]
Let's also consider the pseudohyperbolic distance in $\D$, defined as
\[
 \rho(z,w)=\left|\frac{z-w}{1-\bar z w}\right|\ ,
\]
and the corresponding pseudohyperbolic disks $D(z,r)=\{w\in\D : \rho(z,w)<r\}$.

According to \cite[Theorem 1.2]{HMNT}  $\Lambda\in\Int \N$ if and only if there exists $H\in\Har_+(\D)$ such that 
\begin{equation}\label{intN}
 |B_\lambda(\lambda)|=(1-|\lambda|)|B'(\lambda)|\geq e^{-H(\lambda)},\quad \lil\ .
\end{equation}
Moreover in such case the trace space is
\[
 \N (\Lambda)=\left\{\{\omega(\lambda)\}_{\lil}:\exists H\in \Har_+(\D)\ ,\, \log_+|\omega(\lambda)| \le H(\lambda),\ \lambda\in \Lambda\right\}.
\]
Other properties and characterizations of Nevanlinna interpolating sequences have been given recently in \cite{HMN}.
In these terms $\Lambda\in\Int \N$  when for
every sequence $\omega(\Lambda)\in \N(\Lambda)$ there exists $f\in \N$ such
that $f(\lambda)=\omega(\lambda)$, $\lambda\in\Lambda$. In terms of the
restriction operator 
\[ 
\begin{split}
\mathcal R_\Lambda: \N & \longrightarrow \N (\Lambda)\\
 f\, & \mapsto\; \{{f(\lambda)}\}_{\lambda\in \Lambda},
\end{split} 
\]
$\Lambda$ is  interpolating when 
$\mathcal R_\Lambda(\N)=\N(\Lambda)$.

\begin{definition}
Let $\Lambda$ be a discrete sequence in $\D$ and $\omega$ a function given on
$\Lambda$. The \emph{pseudohyperbolic divided differences of  $\omega$} are defined by induction
as follows 
\[
\begin{split}
\Delta^0 \omega(\lambda_1)  &=\omega(\lambda_1)\ ,\\
\Delta^j\omega(\lambda_1,\ldots,\lambda_{j+1}) 
&=\displaystyle\frac{\Delta^{j-1}\omega(\lambda_2,\ldots,\lambda_{j+1})-\Delta^{
j-1}\omega(\lambda_1,\ldots,\lambda_j)}{b_{\lambda_1}(\lambda_{j+1})}\qquad j\geq
1.\\
\end{split}
\]

For any $n\in \mathbb N$, denote 
\[
\Lambda^n=\{(\lambda_1,\ldots,\lambda_n)\in
\Lambda\times\stackrel{\stackrel{n}{\smile}}{\cdots}\times \Lambda\; :\; 
\lambda_j\not=\lambda_k\ \textrm{if}\  j\not=k\},
\] 
and consider the set $X^{n-1}(\Lambda)$ consisting of the
functions defined in $\Lambda$ with divided differences of order $n-1$ uniformly
controlled by a positive harmonic function $H$ i.e., such that for some $H\in\Har_+(\D)$,
\[
\sup_{(\lambda_1,\ldots,\lambda_n)\in \Lambda^n} \vert
\Delta^{n-1}\omega(\lambda_1,\ldots,\lambda_n)\vert
e^{-[H(\lambda_1)+\cdots+H(\lambda_n)]}<+\infty\ .
\]
\end{definition}

\begin{lemma}\label{inclusions}
Let $n\in\mathbb N$.  For any sequence $\Lambda \subset \D$, we have $X^n(\Lambda)\subset X^{n-1}(\Lambda)\subset\cdots\subset
X^0(\Lambda)=\N(\Lambda)$.
\end{lemma}

\begin{proof}
Assume that $\omega(\Lambda)\in X^n(\Lambda)$, that is, 
\begin{multline*}
\sup_{(\lambda_1,\dots,\lambda_{n+1})\in\Lambda^{n+1}}
 \left|\frac{\Delta^{n-1}\omega(\lambda_2,\dots,\lambda_{n+1})-\Delta^{n-1}
\omega(\lambda_1,\dots,\lambda_{n})}{b_{\lambda_1}(\lambda_{n+1})}\right|
 e^{-[H(\lambda_1)+\cdots+H(\lambda_{n+1})]}<\infty\ .
\end{multline*}
Then, given $(\lambda_1,\dots,\lambda_{n})\in\Lambda^{n}$ and taking
$\lambda_1^0,\dots,\lambda_n^0$ from a finite set (for instance the $n$ first
$\lambda^0_j\in\Lambda$ different of all $\lambda_j$) we have
\begin{multline*}
\Delta^{n-1}\omega(\lambda_1,\dots,\lambda_{n})=
\frac{\Delta^{n-1}\omega(\lambda_1,\dots,\lambda_{n})-\Delta^{n-1}
\omega(\lambda^0_1,\lambda_1,\dots,\lambda_{n-1})}{b_{\lambda_1^0}(\lambda_{n})} b_{\lambda_1^0}(\lambda_{n})+\\
+\frac{\Delta^{n-1} \omega(\lambda^0_1,\lambda_1,\dots,\lambda_{n-1})-\Delta^{n-1} \omega(\lambda^0_2
,\lambda^0_1,\dots,\lambda_{n-2})}{b_{\lambda_2^0}(\lambda_{n-1})}b_{\lambda_2^0}(\lambda_{n-1})+\cdots+\\
\frac{\Delta^{n-1}\omega(\lambda^0_{n-1},\dots,\lambda^0_1,\lambda_{1})-\Delta^{
n-1}\omega(\lambda^0_n,\dots,\lambda^0_1)}{b_{\lambda_n^0}(\lambda_{1})}b_{\lambda_n^0}(\lambda_{1})+
\Delta^{n-1}\omega(\lambda^0_{n},\dots,\lambda^0_1)
\end{multline*}
Since $\omega\in X^{n-1}(\Lambda)$ there exists $H\in\Har_+(\D)$ and a constant $K(\lambda^0_1,\dots,\lambda^0_n)$ such that
\begin{multline*}
\left|\Delta^{n-1}\omega(\lambda_1,\dots,\lambda_{n})\right|\leq  
e^{H(\lambda^0_1)+H(\lambda_1)\cdots+H(\lambda_n)} \rho(\lambda^0_1,\lambda_n)+ 
e^{H(\lambda^0_1)+ H(\lambda_2^0)\cdots+H(\lambda_{n-1})} \rho(\lambda^0_2,\lambda_{n-1})+\\
+\cdots 
+e^{H(\lambda^0_{1})+\cdots+H(\lambda^0_{n})+H(\lambda_{1})}\rho(\lambda^0_n,\lambda_{1})+\Delta^{n-1}\omega(\lambda^0_n,\dots,\lambda^0_1)\\
\leq K(\lambda^0_1,\dots,\lambda^0_n)\, e^{H(\lambda_1)+\cdots+H(\lambda_n)},
\end{multline*}
and the statement follows.
\end{proof}

 The main result of this note is modelled after Vasyunin's description of the
sequences $\Lambda$ in $\D$ such that the trace of the algebra of
bounded holomorphic functions $H^\infty$ on $\Lambda$ equals the space of
pseudohyperbolic divided differences of order $n$ (see \cite{Vas83}, \cite{Vas84}).
Similar results hold also for Hardy spaces (see \cite{BNO} and \cite{H}) and the H\"ormander algebras, both in $\C$ and in $\D$ \cite{MOO}.
The analogue in our context is the following.

\begin{mainT} The identity
$\N| \Lambda= X^{n-1}(\Lambda)$ holds if and only if $\Lambda$ is
the union of $n$ interpolating sequences for $\N$.
\end{mainT}

\section{General properties}

Throughout the proofs we will use repeatedly the well-known \emph{Harnack inequalities}: for $H\in\Har_+(\mathbb D)$ and $z,w\in\mathbb \D$,
\[
 \frac{1-\rho(z,w)}{1+\rho(z,w)}\leq\frac{H(z)}{H(w)}\leq\frac{1+\rho(z,w)}{1-\rho(z,w)}\ .
\]
We shall always assume, without loss of generality, that $H\in\Har_+(\mathbb D)$ is big enough so that for  $z\in D(\lambda, e^{-H(\lambda)})$ the inequalities
$1/2\leq H(z)/H(\lambda)\leq 2$ hold. Actually it is sufficient to assume $\inf\{H(z): z \in \D \} \geq \log 3$.

We begin by showing that one of the inclusions of the Main Theorem is inmediate. 

\begin{proposition}\label{incl}
For all $n\in\mathbb N$,  the inclusion $\N| \Lambda\subset X^{n-1}(\Lambda)$ holds.
\end{proposition}

\begin{proof}
Let $f\in \N$. Let us show  by induction on $j\ge1$ that there exists $H\in\Har_+(\D)$ such that
\[ 
\vert \Delta^{j-1} f(z_1,\ldots,z_j)\vert \le 
e^{H(z_1)+\cdots+H(z_j)}\qquad\textrm{for all $(z_1,\ldots,z_j)\in \D^j$}.
\]

As $f\in \N$, there exists $H\in\Har_+(\D)$ such that $\vert \Delta^0 f(z_1)\vert =\vert f(z_1)\vert \le 
e^{H(z_1)}$, $z_1\in\D$.

Assume that the property is true for $j$ and let $(z_1,\ldots,z_{j+1})\in
\D^{j+1}$. Fix $z_1,\ldots,z_j$ and consider $z_{j+1}$ as the variable in the
function
\[
\Delta^j f(z_1,\ldots, z_{j+1}) =\frac{\Delta^{j-1}
f(z_2,\ldots,z_{j+1})-\Delta^{j-1}f(z_1,\ldots,z_j)}{b_{z_1}(z_{j+1})}.
\]
By the induction hypothesis, there exists $H\in\Har_+(\D)$ such that
\[
\vert \Delta^{j} f(z_1, \ldots,z_{j+1})\vert \le\frac 1{\rho(z_1,z_{j+1})}
 \bigl(e^{H(z_2)+\cdots+H(z_{j+1})}+e^{H(z_1)+\cdots+H(z_j)}\bigr) .
\]

If $\rho(z_1,z_{j+1})\geq 1/2$ we get directly
\[
 \vert \Delta^{j} f(z_1, \ldots,z_{j+1})\vert \le  4 e^{H(z_1)+\cdots+H(z_{j+1})}\ ,
\]
and choosing for instance $\tilde H=H+\log 4$ we get the desired estimate.

If $\rho(z_1,z_{j+1})\leq 1/2$ we apply the maximum principle and Harnack's inequalities
\begin{align*}
\vert \Delta^j f(z_1,\ldots, z_{j+1})\vert  & 
\le \sup_{ \xi : \rho(\xi,z_{j+1}) =1/2} \vert \Delta^j f(z_1,\ldots, z_j, \xi_{j+1})\vert\\
&\leq
\sup_{\xi : \rho(\xi,z_{j+1}) =1/2} 4 e^{H(z_1)+\cdots+H(z_j)+H(\xi)}\\ 
&  \le 4e^{2[H(z_1)+\cdots+H(z_j)+H(z_{j+1})]}.
\end{align*}
Choosing here $\tilde H=2H+\log4$ we get the desired estimate.
\end{proof}

\begin{definition}
A sequence $\Lambda$ is \emph{weakly separated} if there exists $H\in\Har_+(\D)$
such that the disks $D(\lambda, e^{-H(\lambda)})$, $\lambda\in\Lambda$, are pairwise disjoint.
\end{definition}

\begin{remark}
If $\Lambda$ is weakly separated then $X^0(\Lambda)=X^n(\Lambda)$, for all $n\in\mathbb
N$. 

By Lemma~\ref{inclusions}, to see this it is enough to prove (by induction) that $X^0(\Lambda)\subset X^n(\Lambda)$ for all $n\in\mathbb N$.

For $n=0$ this is trivial. 

Assume now that $X^0(\Lambda)\subset X^{n-1}(\Lambda)$ and take
$\omega(\Lambda)\in X^0(\Lambda)$. Since $\rho(\lambda_1,\lambda_{n+1})\geq e^{-H_0(\lambda_1)}$ for some $H_0\in\Har_+(\D)$ we have
\begin{align*}
\left|\Delta^n\omega(\lambda_1,\dots,\lambda_{n+1})\right|&=
\left|\frac{\Delta^{n-1}\omega(\lambda_2,\dots,\lambda_{n+1})-\Delta^{n-1}\omega(\lambda_1,\dots,\lambda_n)}{b_{\lambda_1}(\lambda_{n+1})}\right|\\
&\leq e^{H_0(\lambda_1)} \left( e^{H(\lambda_2)+\cdots+H(\lambda_{n+1})} + e^{H(\lambda_1)+\cdots+H(\lambda_{n})}\right)
\end{align*}
for some $H \in \Har_+ (\D)$. Taking $\tilde {H}=H+H_0$ we are done.
\end{remark}

\begin{lemma}\label{equiv}
Let $n\geq 1$. The following assertions are equivalent:
\begin{itemize}
\item[(a)] $\Lambda$ is the union of $n$  weakly separated sequences,

\item[(b)] There exist $H\in\Har_+(\D)$ such that
\[
 \sup_{\lambda\in\Lambda} \# [\Lambda\cap D(\lambda, e^{-H(\lambda)})]\leq n\ .
\]

\item[(c)] $X^{n-1}(\Lambda)=X^n(\Lambda)$.
\end{itemize}

\end{lemma}

\begin{proof}
(a) $\Rightarrow$(b). This is clear, by the weak separation.
 
(b) $\Rightarrow$(a). We proceed by induction on $j=1,\ldots,n$. For $j=1$, it
is again clear by the definition of weak separation. Assume  the property true
for $j-1$. Let $H\in\Har_+(\D)$ , $\inf\{H(z): z \in \D \} \geq \log 3$, be such that
$\sup_{\lambda\in\Lambda} \# [\Lambda\cap D(\lambda, e^{-H(\lambda)})]\leq j$.
We split the sequence $\Lambda=\Lambda_a\cup\Lambda_b$ where
\begin{eqnarray*}
 \Lambda_a&=&\bigcup_{\{\lambda\in\Lambda:\#(\Lambda\cap D(\lambda,e^{-10H(\lambda)}))= j\}}
 (\Lambda\cap D(\lambda,e^{-10H(\lambda)}))\\
  \Lambda_b&=&\Lambda\setminus\Lambda_a  
\end{eqnarray*}
Now, for every $\lambda\in\Lambda_b$, we have $\#(\Lambda\cap D(\lambda,e^{-10H(\lambda)}))\le j-1$,
and by the induction hypothesis, $\Lambda_b$ splits into $j-1$ separated sequences $\Lambda_1,
\ldots,\Lambda_{j-1}$. 


In the case $\lambda\in\Lambda_a$, there is obviously no point in the annulus
$D(\lambda,e^{-H(\lambda)})\setminus D(\lambda,e^{-10H(\lambda)})$ which means that the
$j$ points in $D(\lambda,e^{-10H(\lambda)}))$ are far from the other points of $\Lambda$. So we can add each one of these $j$ points in a weakly separated way to one of the sequences $\Lambda_1,\ldots,\Lambda_{j-1}$,
and the $j$-th point in a new sequence $\Lambda_j$ (which is of course weakly separated since the 
groups $\Lambda\cap D(\lambda,e^{-10H(\lambda)})$ appearing in $\Lambda_a$ are weakly separated).


(b)$\Rightarrow$(c). It remains to see that $X^{n-1}(\Lambda)\subset
X^{n}(\Lambda)$.
Given $\omega(\Lambda)\in X^{n-1}(\Lambda)$ and points
$(\lambda_1,\ldots,\lambda_{n+1}) \in \Lambda^{n+1}$, we have to estimate
$\Delta^n \omega(\lambda_1,\ldots,\lambda_{n+1})$. Under the assumption (b), at
least one of these $n+1$ points  is not in the disk $D(\lambda_1, e^{-H(\lambda_1)})$. 
Note that $\Lambda^n$ is invariant by permutation of the
$n+1$ points, thus we may assume that $\rho(\lambda_1 ,\lambda_{n+1}) \ge
e^{-H(\lambda_1)}$. Using the fact that $\omega(\Lambda)\in
X^{n-1}(\Lambda)$, there exists $H_0\in\Har_+(\D)$ such that
\[
\begin{split}
\vert \Delta^n \omega (\lambda_1,\ldots,\lambda_{n+1})\vert & \le 
\frac{\vert\Delta^{n-1} \omega (\lambda_2,\ldots,\lambda_{n+1})\vert+\vert \Delta^{n-1}
\omega (\lambda_1,\ldots,\lambda_n)\vert}{\rho(\lambda_1 ,\lambda_{n+1})}
\\
& \le  e^{H(\lambda_1)} \left( e^{H_0(\lambda_2)+\cdots+H_0(\lambda_{n+1})} +e^{H_0(\lambda_1)+\cdots+H_0(\lambda_{n})}\right)\\
& \le 2 e^{H(\lambda_1)} e^{H_0(\lambda_1)+\cdots+H_0(\lambda_{n+1})}\ .
\end{split}
\] 
Taking $\tilde H=H_0+H + \log 2$ we get the desired estimate.

(c)$\Rightarrow$(b). We prove this by contraposition. Assume that for all $H\in\Har_+(\D)$ there exists $\lil$ such that
\begin{eqnarray}\label{many}
\# [\Lambda\cap D(\lambda, e^{-H(\lambda)})]>n\ . 
\end{eqnarray}
Consider the partition of $\mathbb D$ into the dyadic squares
\[
 Q_{k,j}=\bigl\{z=re^{i\theta}\in\mathbb D : 1-2^{-k}\leq r< 1-2^{-k-1}\ ,\ j\frac{2\pi}k\leq \theta<(j+1)\frac{2\pi}k\bigr\},
\]
where $k\geq 0$ and $ j=0,\dots 2^k-1$.

Let $\Lambda_{k,j}=\Lambda\cap Q_{k,j}$ and
\[
 r_{k,j}=\inf\{r>0\ :\ \exists \lambda\in\Lambda_{k,j}\ :\ \#(\Lambda\cap \overline{D(\lambda,r)})\ge n+1\}.
\]
Take $\alpha_{k,j}\in\Lambda_{k,j}$ such that $\#(\Lambda\cap \overline{D(\alpha_{k,j},r_{k,j})})\ge n+1$.

\medskip

\textit{Claim:} For all $H\in\Har_+(\D)$,
\[
 \inf_{k,j}\frac{r_{k,j}}{e^{-H(\alpha_{k,j})}}=0\ .
\]

To see this assume otherwise that there exist $H\in\Har_+(\D)$  and $\eta >0$ with 
\[
 \frac{r_{k,j}}{e^{-H(\alpha_{k,j})}}\geq \eta \ .
\]
In particular, by Harnack's inequalities,
\begin{equation}\label{claim}
 \log \frac 1{r_{k,j}}\leq 3 H(z)+\log(\frac 1\eta),\quad z\in Q_{k,j}.
\end{equation}

Let $\tilde H:=\log (2/\eta)+ 4 H \in \Har_+(\D)$. By the hypothesis \eqref{many} there exist $k_0\geq 0$, $j_0\in\{0,\dots,2^{k_0}-1\}$, $\lambda_{k_0, j_0}\in\Lambda_{k_0, j_0}$ such that
\[
 \#\left[\Lambda\cap\overline{D(\lambda_{k_0,j_0}, e^{-\tilde H(\lambda_{k_0,j_0})})}\right]\geq n+1\ .
\]
In particular, by definition of $r_{k,j}$, we have $r_{k_0,j_0}\le e^{-\tilde H(\lambda_{k_0, j_0})}$, that is
\[
 \log\frac 1{r_{k_0,j_0}}\geq \tilde H(\lambda_{k_0, j_0})= \log(\frac 2\eta)+ 4 H(\lambda_{k_0,j_0}) ,
\]
which contradicts \eqref{claim}.

Now take a separated sequence $\mathcal L\subset \{\alpha_{k,j}\}_{k,j}$ for which the disks $D(\alpha,r_{\alpha})$, $\alpha\in\mathcal{L}$, are disjoint, where for $\alpha=\alpha_{k,j}\in \mathcal L$
we denote $r_{\alpha}=r_{k,j}$.
Given $\alpha\in\mathcal L$, let $\lambda_1^{\alpha},\dots,\lambda_n^{\alpha}$ be its $n$ nearest (not necessarily unique) points, arranged by increasing distance.
Notice that $\rho(\alpha,\lambda_n^{\alpha})=r_{\alpha}$.

In order to construct a sequence $\omega(\Lambda)\in X^{n-1}(\Lambda)\setminus
X^n(\Lambda)$, put 
\[
\begin{cases}
\omega(\alpha)  = \prod\limits_{j=1}^{n-1} b_{\alpha}(\lambda_j^{\alpha}),\quad  &\textrm{for all
} \alpha\in\mathcal L\\
\omega(\lambda)  = 0\ & \textrm{ if }\, \lambda\in \Lambda\setminus \mathcal L.
\end{cases}
\]
 
To see that $\omega(\Lambda)\in X^{n-1}(\Lambda)$ let us estimate $\Delta^{n-1}\omega(\lambda_1,\ldots,\lambda_n)$ for any given
$(\lambda_1,\ldots,\lambda_n)\in \Lambda^n$. 
By the separation conditions on $\mathcal L$, we know that none of the 
$\lambda_j^{\alpha}$ is in $\mathcal L$. Hence, we may
assume that at most one of the points is in $\mathcal L$. On the other hand, it
is clear that 
$\Delta^{n-1}\omega(\lambda_1,\ldots,\lambda_n)=0$ if all the points are in
$\Lambda\setminus \mathcal L$. Thus, taking into account that $\Delta^{n-1}$ is
invariant by permutations, we will only consider the case where  $\lambda_n$ is
some $\alpha\in \mathcal L$ and  $\lambda_1,\ldots,\lambda_{n-1}$ are in
$\Lambda\setminus \mathcal L$. In that case, 
\[ 
\vert \Delta^{n-1}\omega(\lambda_1,\ldots,\lambda_{n-1},\alpha)\vert =\vert
\omega(\alpha)\vert \prod_{j=1}^{n-1} \rho(\alpha,\lambda_j)^{-1}=\prod_{j=1}^{n-1}
\frac{\rho(\alpha,\lambda_j^{\alpha})}{\rho(\alpha,\lambda_j)}
\le 1,
\]
 as desired.

On the other hand, a similar computation yields
\[ 
\begin{split}
\vert \Delta^n \omega(\lambda_1^{\alpha},\ldots,\lambda_n^{\alpha},\alpha)\vert   =\vert
\omega(\alpha)\vert \prod_{j=1}^n\rho(\alpha,\lambda_j^{\alpha})^{-1}
 =\rho(\alpha,\lambda_n^{\alpha})^{-1} =r_{\alpha}^{-1}.
\end{split}
\]
The Claim above prevents the existence of $H\in\Har_+(\D)$ such that
\[
r_{\alpha}^{-1}=\vert \Delta^n \omega(\lambda_1^{\alpha},\ldots,\lambda_n^{\alpha},\alpha)\vert
e^{-(H(\lambda_1^{\alpha})+\cdots+H(\lambda_n^{\alpha})+H(\alpha))}\leq C\ ,
\]
since otherwise, again by Harnack's inequalities, we would have
\[
 r_{\alpha}^{-1}\leq e^{3(n+1) H(\alpha)}, \quad \alpha\in\mathcal{L}\ .
\]
\end{proof}

It is clear from the characterization \eqref{intN} of interpolating sequences for $\N$ that such sequences must be weakly separated. The previous result gives another way of showing it. 

\begin{corollary}\label{weak}
If $\Lambda$ is an interpolating sequence, then it is weakly separated.
\end{corollary}

\begin{proof}
If $\Lambda$ is an interpolating sequence, then $\N|\Lambda=X^0(\Lambda)$. 
On the other hand, by Proposition \ref{incl}, $\N|\Lambda\subset X^1(\Lambda)$. 
Thus  $X^0(\Lambda)= X^1(\Lambda)$.
We conclude by the preceding lemma applied to the particular case $n=1$.
\end{proof}

The covering provided by the following result will be useful.

\begin{lemma}\label{cover}
Let $\Lambda_1,\ldots,\Lambda_n$ be weakly separated sequences. 
There exist $H\in\Har_+(\D)$,  positive constants $\alpha,\beta$, a
subsequence $\mathcal{L}\subset \Lambda_1\cup\cdots\cup \Lambda_n$ and disks
$D_\lambda=D(\lambda, r_\lambda)$, $\lambda\in \mathcal{L}$, such that
\begin{itemize}
\item[(i)]  $\Lambda_1\cup\cdots\cup \Lambda_n\subset  \cup_{\lambda\in \mathcal{L}} D_\lambda$, 
\smallskip
\item[(ii)]   $e^{-\beta H(\lambda)}\le r_\lambda\le e^{-\alpha H(\lambda)}$ for all $\lambda \in \mathcal{L}$,
\smallskip
\item[(iii)]  $\rho(D_\lambda,D_{\lambda'})\ge  e^{-\beta H(\lambda)}$ for all $\lambda,\lambda' \in \mathcal{L}$,
$\lambda\neq \lambda'$.
\smallskip
\item[(iv)] $\#(\Lambda_j\cap D_\lambda)\leq 1$ for all $j=1,\dots,n$ and
$\lambda\in \mathcal{L}$. 
\end{itemize}
\end{lemma}

\begin{proof}
Let $H\in\Har_+(\D)$ be such that
\begin{equation}\label{separation}
 \rho(\lambda,\lambda') \ge  e^{-H(\lambda)}, \qquad
\forall\lambda,\lambda'\in \Lambda_j,\ \lambda\neq\lambda', \  \forall
j=1,\ldots, n\ .
\end{equation}

We will proceed by induction on $k=1,\ldots,n$ to show the existence of a
subsequence $\mathcal{L}_k\subset \Lambda_1\cup\cdots\cup \Lambda_{k}$ such that:
\begin{align*}
&(i)_k \quad \Lambda_1\cup\cdots\cup \Lambda_k \subset  
\cup_{\lambda\in\mathcal{L}_k} D(\lambda, R_\lambda^k), \\
&(ii)_k \quad e^{-\beta_k H(\lambda)} \le R_\lambda^k\le e^{-\alpha_k H(\lambda)}, \\
&(iii)_k\quad \rho(D(\lambda, R_\lambda^k),D(\lambda', R_{\lambda'}^k))
\ge  e^{-\beta_k H(\lambda)}\, \textrm{for any $\lambda,\lambda'\in \mathcal{L}_k$, $\lambda\neq\lambda'$}.
\end{align*}

Then it suffices to chose $\mathcal{L}=\mathcal{L}_n$, $\alpha=\alpha_n$, $\beta=\beta_n$, $r_\lambda=R_\lambda^n$. 
The weak separation and the fact that $r_\lambda< e^{-H(\lambda)}/3 $ implies that 
$\#\Lambda_j\cap D(\lambda,r_\lambda)\leq 1$, $j=1,\dots,k$, hence the lemma follows.

For $k=1$, the property is clearly verified with  $\mathcal{L}_1=\Lambda_1$ and
$R_\lambda^1 = e^{-CH(\lambda)}$, with $C$ big enough so that $(iii)_1$ holds ($C=3$, for instance). Properties $(i)_1$, $(ii)_1$ follow immediately.

Assume the property true for $k$ and split $\mathcal{L}_k=\mathcal{M}_1\cup
\mathcal{M}_2$ and $\Lambda_{k+1}=\mathcal{N}_1\cup \mathcal{N}_2$, where 
\[
\begin{split}
\mathcal{M}_1 & =\{\lambda\in \mathcal{L}_k\ :\ D(\lambda, R_\lambda^k+1/4\, e^{-\beta_k H(\lambda)})\cap \Lambda_{k+1}\not= \emptyset\}, \\
\mathcal{N}_1 & =\Lambda_{k+1}\cap\bigcup_{\lambda\in \mathcal{L}_k} 
D(\lambda, R_\lambda^k+1/4\, e^{-\beta_k H(\lambda)}), \\
\mathcal{M}_2 & = \mathcal{L}_k\setminus \mathcal{M}_1,\\
\mathcal{N}_2 & = \Lambda_{k+1}\setminus \mathcal{N}_1.
\end{split}
\]
Now, we put $\mathcal{L}_{k+1}=\mathcal{L}_k\cup \mathcal{N}_2$ and define the
radii $R_\lambda^{k+1}$ as follows: 
\[
R_\lambda^{k+1}  =
\begin{cases}
R_\lambda^k+1/4\, e^{-\beta_k H(\lambda)}\ &\textrm{if}\ \lambda\in  \mathcal{M}_1,   \\
R_\lambda^k \ &\textrm{if}\  \lambda\in  \mathcal{M}_2, \\
 1/8\, e^{-\beta_{k} H(\lambda)} \ &\textrm{if}\  \lambda\in  \mathcal{N}_2.
\end{cases}
\]
It is clear that $(i)_{k+1}$ holds:
 \[
\Lambda_1\cup\cdots\cup \Lambda_{k+1}\subset  \bigcup_{ \lambda\in
\mathcal{L}_{k+1}}
D(\lambda, R_\lambda^{k+1})\ .
\]
Also, by the induction hypothesis,
\[
\frac 18  e^{-\beta_{k} H(\lambda)} \le  R_\lambda^{k+1}\le e^{-\alpha_k H(\lambda)}+\frac 14 e^{-\beta_k H(\lambda)}.
\]
Thus, to see $(ii)_{k+1}$ there is enough to choose $\alpha_{k+1},\beta_{k+1}$ such that
\[
 e^{-\alpha_k H(\lambda)}+1/4\, e^{-\beta_k H(\lambda)}\leq e^{-\alpha_{k+1} H(\lambda)},
\]
for instance $\alpha_{k+1}=\alpha_{k}-1$, and
\begin{equation}\label{B}
 1/8\, e^{-\beta_k H(\lambda)} \geq e^{-\beta_{k+1} H(\lambda)}\ ,
\end{equation}
that is $\beta_{k+1} H(\lambda)\geq \beta_{k} H(\lambda)+\log 8$. Assuming without loss of generality that $H(\lambda)\geq \log 8$, there is enough choosing $\beta_{k+1}\geq \beta_k+1$.

In order to prove $(iii)_k$ take now  $\lambda,\lambda'\in\mathcal{L}_{k+1}$, $\lambda\neq\lambda'$. Notice that 
\[
\rho(D(\lambda,R_\lambda^{k+1}), D(\lambda',R_{\lambda'}^{k+1}))=
\rho( \lambda,\lambda')-R_\lambda^{k+1}-R_{\lambda'}^{k+1} .
\]

Split into four different cases:

1. \underline{$\lambda,\lambda'\in\mathcal{L}_k$}. Assume without loss of generality that $H(\lambda)\leq H(\lambda')$. Then, by the definition of $R_\lambda^{k+1}$, we see that
\[
 \rho(D(\lambda,R_\lambda^{k+1}), D(\lambda',R_{\lambda'}^{k+1}))=\rho( \lambda,\lambda')-R_\lambda^{k}-R_{\lambda'}^{k}-\frac 14 e^{-\beta_k H(\lambda)}-\frac 14 e^{-\beta_k H(\lambda')}.
\]
By inductive hypothesis
\[
 \rho( \lambda,\lambda')-R_\lambda^{k}-R_{\lambda'}^{k}=\rho(D(\lambda,R_\lambda^{k}), D(\lambda',R_{\lambda'}^{k}))\geq e^{-\beta_k H(\lambda)}\ .
\]
Thus, by \eqref{B},
\begin{align*}
\rho(D(\lambda,R_\lambda^{k+1}), D(\lambda',R_{\lambda'}^{k+1}))&
\ge e^{-\beta_k H(\lambda)}-\frac 12 e^{-\beta_k H(\lambda)}=\frac 12 e^{-\beta_k H(\lambda)}\geq e^{-\beta_{k+1} H(\lambda)} .
\end{align*}

2. \underline{$\lambda,\lambda'\in\mathcal{N}_2$}. Assume also $H(\lambda)\leq H(\lambda')$. 
Condition \eqref{separation} implies $\rho( \lambda,\lambda')\ge  e^{-H(\lambda)}$, hence
 \[
\rho(D(\lambda,R_\lambda^{k+1}), D(\lambda',R_{\lambda'}^{k+1}))
\ge  e^{-H(\lambda)}-\frac 14 e^{-\beta_k H(\lambda)} .
\]
If $\beta_k\geq 2$, by \eqref{B} we have
\[
 \rho(D(\lambda,R_\lambda^{k+1}), D(\lambda',R_{\lambda'}^{k+1}))\ge e^{-2 H(\lambda)}\geq  e^{-\beta_k H(\lambda)}\ge e^{-\beta_{k+1} H(\lambda)}.
\]

3. \underline{$\lambda \in \mathcal{M}_1, {\lambda'}\in \mathcal{N}_2$} 
By definition of $\mathcal{M}_1$ there exists $\beta\in \mathcal{N}_1 $ such that 
\[
 \rho( \lambda,\beta)\le R_\lambda^{k} +\frac 14 e^{-\beta_k H(\lambda)}.
\]

Then, using \eqref{separation} on $\beta,\lambda'\in\Lambda_{k+1}$,  we have, by Harnack's inequalities (if $\beta_k\geq 4$),
\begin{align*}
\rho(\lambda,\lambda') &\ge \rho( \beta,\lambda')-\rho(\lambda,\beta)
\ge  e^{-H(\beta)}-R_\lambda^{k}-\frac 14 e^{-\beta_k H(\lambda)} 
\ge  e^{-2 H(\lambda)}-\frac 54 e^{-\beta_k H(\lambda)} \\
&\ge  e^{-4 H(\lambda)}\geq e^{-\beta_k H(\lambda)}\geq e^{-\beta_{k+1} H(\lambda)}
\ .
\end{align*}

4. \underline{$\lambda \in \mathcal{M}_2, {\lambda'}\in \mathcal{N}_2$}. 
Taking into account the definition of $R_\lambda^{k+1},R_{\lambda'}^{k+1}$ we have
\begin{align*}
\rho(D(\lambda,R_\lambda^{k+1}), D(\lambda',R_{\lambda'}^{k+1}))  =
\rho(\lambda,\lambda')-R_\lambda^k-\frac 18  e^{-\beta_k H(\lambda)}
\end{align*}

Since
\[
 \rho(\lambda,\lambda')-R_\lambda^k\geq \rho(D(\lambda,R_\lambda^{k}), D(\lambda',R_{\lambda'}^{k})),
\]
by inductive hypothesis and by \eqref{B}
\begin{align*}
\rho(D(\lambda,R_\lambda^{k+1}), D(\lambda',R_{\lambda'}^{k+1}))  \geq 
\frac 14  e^{-\beta_k H(\lambda)}-\frac 18  e^{-\beta_k H(\lambda)}\geq 
e^{-\beta_{k+1} H(\lambda)}\ .
\end{align*}

All together, it is enough to start with $C>n$, define $\alpha_1=\beta_1=C$, and then define $\alpha_k$, $\beta_k$ inductively by
\[
 \alpha_{k+1}=\alpha_k-1=\cdots= C-k\ ,\qquad \beta_{k+1}=\beta_k+1=\cdots= C+k\ .
\]
\end{proof}

\section{Proof of Main Theorem. Necessity}\label{necessity}

Assume $\N|\Lambda= X^{n-1}(\Lambda)$, $n\ge 2$. Using
Proposition \ref{incl}, we have $X^{n-1}(\Lambda)=X^n(\Lambda)$, and by Lemma
\ref{equiv} we deduce that 
$\Lambda=\Lambda_1\cup \cdots\cup \Lambda_n$, where $\Lambda_1,\ldots,\Lambda_n$ are weakly separated sequences. We want to show that each $\Lambda_j$ is an interpolating sequence. 

Let $\omega(\Lambda_j)\in \N(\Lambda_j)= X^0(\Lambda_j)$.  
Let $\cup_{\lambda\in \mathcal{L}}D_\lambda$ be the covering of $\Lambda$ given
by Lemma \ref{cover}. 
We extend $\omega(\Lambda_j)$ to a sequence $\omega(\Lambda)$ which is constant
on each $D_\lambda\cap\Lambda_j$ in the following way:
\[
\omega_{|D_\lambda\cap\Lambda }=
\begin{cases}
0\quad &\textrm{if $D_\lambda\cap\Lambda_j=\emptyset$}\\
\omega(\alpha)\quad &\textrm{if $D_\lambda\cap\Lambda_j=\{\alpha\}$}\ .
\end{cases}
\]
We verify by induction that the extended sequence is in $X^{k-1}(\Lambda)$ for
all $k\leq n$.  It is clear that it belongs to $X^0(\Lambda)$. 

Assume that $\omega\in X^{k-2}(\Lambda)$, $k\geq 2$, and consider $(\alpha_1,\ldots,\alpha_k)\in
\Lambda^k$. If all the points are in the same $D_\lambda$ then
$\Delta^{k-1}\omega(\alpha_1,\ldots,\alpha_k)=0$, so we may assume that
$\alpha_1\in D_{\lambda}$ and $\alpha_k\in D_{\lambda'}$ with $\lambda
\not=\lambda'$. Then  we have, for some $H_0\in\Har_+(\D)$,
 \[
 \rho( \alpha_1,\alpha_k) \ge  e^{-\beta H_0(\alpha_1)},\qquad k\neq 1.
 \]
With this and the induction hypothesis it is clear that for some $H\in\Har_+(\D)$,
\begin{align*}
\vert \Delta^{k-1}\omega(\alpha_1,\ldots,\alpha_k)\vert&=\left|\frac{
\Delta^{k-2}\omega(\alpha_2,\ldots,\alpha_k)-
\Delta^{k-2}\omega(\alpha_1,\ldots,\alpha_{k-1})}{b_{\alpha_1}(\alpha_k)}\right|\\
&\le e^{\beta H_0(\alpha_1)} \bigl(e^{H(\alpha_2)+\cdots+H(\alpha_k)}+e^{H(\alpha_1)+\cdots+H(\alpha_{k-1})}\bigr)\ .
\end{align*}
Taking for instance $\tilde H=H+\beta H_0 + \log 2$ we get 
\[
 \vert \Delta^{k-1}\omega(\alpha_1,\ldots,\alpha_k)\vert\leq e^{\tilde H(\alpha_1)+\cdots+\tilde H(\alpha_k)}\ ,
\]
thus $\omega(\Lambda)\in X^{k-1}(\Lambda)$. By assumption there
exist $f\in \N$ interpolating  the values $\omega(\Lambda)$. In particular $f$
interpolates  $\omega(\Lambda_j)$.

\section{Proof of the Main Theorem. Sufficiency}\label{sufficiency}

Assume $\Lambda= \Lambda_1\cup\dots\cup \Lambda_n$, where $\Lambda_j\in\Int \N$, $j=1,\dots,n$, and denote $\Lambda_j=\{\lambda_k^{(j)}\}_{k\in\mathbb N}$. Denote also by $B_j$ the Blaschke product with zeros on $\Lambda_j$. We will use the following property of the Nevanlinna interpolating sequences (see Theorem~1.2 in \cite{HMN}).

\begin{lemma}\label{int-local}
Let $\Lambda\in\Int N $ and let $B$ the Blaschke product associated to $\Lambda$. There exists $H_1\in\Har_+(\D)$ such that 
\[
|B(z)|\geq e^{-H_1(z)}\rho(z,\Lambda)\ \qquad z\in\D\ .
\]
\end{lemma}
 
According to Proposition~\ref{incl} we only need to see that $X^{n-1}(\Lambda)\subset \N|\Lambda$. 
Let then $\omega(\Lambda)\in X^{n-1}(\Lambda)$ and split it
\[
 \{\omega(\lambda)\}_{\lambda\in\Lambda}= \{\omega_k^{(1)}\}_{k\in\mathbb N}\cup\dots\cup \{\omega_k^{(n)}\}_{k\in\mathbb N}\ ,
\]
where $\omega_k^{(j)}=\omega(\lambda_k^{(j)})$, $j=1,\dots,n$, $k\in\mathbb N$. By Lemma~\ref{inclusions} and the hypothesis $\{\omega_k^{(1)}\}_{k\in\mathbb N}\in X^0(\Lambda_1)$, hence there exists $f_1\in\N$ such that
\[
 f_1(\lambda_k^{(1)})=\omega_k^{(1)}\ ,\qquad k\in\mathbb N\ .
\]
In order to interpolate also the values $\{\omega_k^{(2)}\}_{k}$ consider functions of the form
\[
 f_2(z)=f_1(z)+B_1(z) g_2(z)\ .
\]
Immediately $f_2(\lambda_k^{(1)})=f_1(\lambda_k^{(1)})=\omega_k^{(1)}$, $k\in\mathbb N$, and we will have $f_2(\lambda_k^{(2)})=\omega_k^{(2)}$ as soon as we find $g_2\in\N$ such that
\[
 g_2(\lambda_k^{(2)})=\frac{\omega_k^{(2)}-f_1(\lambda_k^{(2)})}{B_1(\lambda_k^{(2)})}\ , k\in\N\ .
\]
Since $\Lambda_2\in\Int \N$ such $g_2$ will exist as soon as the sequence in the right hand side is majorized by a sequence of the form 
$\{e^{H(\lambda_k^{(2)})}\}_k$. 

Given $\lambda_k^{(2)}\in\Lambda_2$ pick $\lambda_k^{(1)}$ such that $\rho(\lambda_k^{(2)},\Lambda_1)=\rho(\lambda_k^{(2)},\lambda_k^{(1)})$. There is no restriction in assuming that $\rho(\lambda_k^{(2)},\lambda_k^{(1)})\leq 1/2$. Then, by Lemma~\ref{int-local} there exists $H_1\in \Har_+(\D)$ such that 
\[
 |B_1(\lambda_k^{(2)})|\ge e^{-H_1(\lambda_k^{(2)})}\rho(\lambda_k^{(1)},\lambda_k^{(2)})\qquad k\in\mathbb N .
\]
Now, since $f_1(\lambda_{k}^{(1)})=\omega_{k}^{(1)}$ we have
\begin{align*}
 \left|\frac{\omega_k^{(2)}-f_1(\lambda_k^{(2)})}{B_1(\lambda_k^{(2)})}\right| &\leq
 \left|\frac{\omega_k^{(2)}-\omega_{k}^{(1)}}{B_1(\lambda_k^{(2)})}\right|
 +\left|\frac{f_1(\lambda_{k}^{(1)})-f_1(\lambda_k^{(2)})}{B_1(\lambda_k^{(2)})}\right|\\
 &\leq \left(\Delta^1(\omega_{k}^{(1)},\omega_k^{(2)}) + \Delta^1( f_1(\lambda_{k}^{(1)}),f_1(\lambda_k^{(2)}))\right) e^{H_1(\lambda_k^{(2)})}\ .
\end{align*}
By hypothesis, and since $f_1 \in\N$, there exists $H_2\in\Har_+(\D)$ such that
\[
\Delta^1(\omega_{k}^{(1)},\omega_k^{(2)}) + \Delta^1( f_1(\lambda_{k}^{(1)}),f_1(\lambda_k^{(2)}))
 \leq e^{H_2(\lambda_{k}^{(1)})+H_2(\lambda_k^{(2)})},
\]
and therefore, by Harnack's inequalities,
\[
 \left|\frac{\omega_k^{(2)}-f_1(\lambda_k^{(2)})}{B_1(\lambda_k^{(2)})}\right|\leq e^{H_2(\lambda_{k}^{(1)}) +H_2(\lambda_{k}^{(2)})} e^{H_1(\lambda_{k}^{(2)}) }
\leq  e^{3(H_1+H_2)(\lambda_k^{(2)})} ,
\]
In general, assume that we have $f_{n-1}\in\N$ such that
\[
 f_{n-1}(\lambda_k^{(j)})=\omega_k^{(j)}\, \qquad k\in\mathbb N,\ j=1,\dots,n-1\ .
\]
We look for a function $f_n\in\N$ interpolating the whole $\Lambda$ of the form
\[
 f_n=f_{n-1}+B_1\cdots B_{n-1} g_n\ .
\]
We need then $g_n \in\N$ with
\[
 g_n(\lambda_{k}^{(n)})=\frac{\omega_{k}^{(n)}-f_{n-1}(\lambda_{k}^{(n)})}{B_1(\lambda_{k}^{(n)})\cdots B_{n-1}(\lambda_{k}^{(n)})},\qquad k\in\mathbb N\ .
\]
Let us see that the sequence of values in the right hand side of this identity have a majorant of the form $\{e^{H(\lambda_{k}^{(n)})}\}_k$.

Pick $\lambda_{k}^{(j)}\in\Lambda_j$, $j=1,\dots,n-1$ such that $\rho(\lambda_{k}^{(n)},\Lambda_j )=\rho(\lambda_{k}^{(n)},\lambda_{k}^{(j)})$. There is no restriction in assuming that $\rho(\lambda_{k}^{(n)},\lambda_{k}^{(j)})\leq 1/2$.
Since $f_{n-1}(\lambda_{k}^{(j)})=\omega_{k}^{(j)}$, $j=1,\dots,n-1$, an immediate computation shows that
\begin{multline*}
 \omega_{k}^{(n)}-f_{n-1}(\lambda_{k}^{(n)})=\left[\Delta^{n-1}(\omega_{k}^{(1)},\dots,\omega_{k}^{(n-1)},\omega_{k}^{(n)})- \right.\\
 \left. - \Delta^{n-1}(f_{n-1}(\lambda_{k}^{(1)}),\dots,f_{n-1}(\lambda_{k}^{(n-1)}), f_{n-1}(\lambda_{k}^{(n)}))\right]\, 
 b_{\lambda_{k}^{(1)}}(\lambda_{k}^{(n)})\cdots b_{\lambda_{k}^{(n-1)}}(\lambda_{k}^{(n)})\ .
\end{multline*}
Again by Lemma~\ref{int-local}, there exists $H_1\in\Har_+(\D)$ such that
\[
 |B_j(\lambda_{k}^{(n)})|\geq e^{-H_1(\lambda_{k}^{(n)})} \rho(\lambda_{k}^{(j)}, \lambda_{k}^{(n)})\ , k\in\mathbb N,\ j=1,\dots,n-1 .
\]
Hence, by hypothesis and the fact that $f_{n-1}\in\N$ there exists $H\in\Har_+(\D)$ such that
\begin{multline*}
 \left|\frac{\omega_{k}^{(n)}-f_{n-1}(\lambda_{k}^{(n)})}{B_1(\lambda_{k}^{(n)})\cdots B_{n-1}(\lambda_{k}^{(n)})}\right|\leq
 [| \Delta^{n-1}(\omega_{k}^{(1)},\dots,\omega_{k}^{(n)}) |
 + | \Delta^{n-1}(f_{n-1}(\lambda_{k}^{(1)}),\dots, f_{n-1}(\lambda_{k}^{(n)}))|]\, e^{(n-1) H_1(\lambda_{k}^{(n)})}\\
 \leq e^{H(\lambda_{k}^{(1)})+\cdots+ H(\lambda_{k}^{(n-1)})+H(\lambda_{k}^{(n)})+(n-1) H_1(\lambda_{k}^{(n)})}\ .
\end{multline*}

Finally, by Harnack's inequalities, this is bounded by $e^{2n (H(\lambda_{k}^{(n)})+ H_1(\lambda_{k}^{(n)}))}$.

\end{document}